\documentclass{article}
\usepackage{amsthm}
\usepackage{amsmath}

\newtheorem{theorem}{Theorem}
\newtheorem{lem}{Lemma}
\newtheorem{prop}{Proposition}
\newtheorem{problem}{Problem}
\newtheorem{cor}{Corollary}
\newtheorem{claim}{Claim}

\begin{document}

\title{The $L^p$ CR Hartogs separate analyticity theorem for convex domains}         % Enter your title between curly braces
\author{Mark G. Lawrence}        % Enter your name between curly braces

\maketitle

\begin{abstract}
In this note, a very general theorem of the CR Hartogs type is proved for almost generic strictly convex domains in $\bf C^n$ with real analytic boundary. Given such a domain $D$, and given an $L^p$ function $f$ on $\partial D$ which has holomorphic extensions on the slices of $D$ by complex lines parallel to the coordinate axes, $f$ must be CR---i.e. $f$ has a holomorphic extension to $D$ which is in the Hardy space $H^p(D)$. As a corollary, the Szeg\H{o} kernel is   shown to be the strong operator limit of $(\pi_1 \pi_2)^n$, where $\pi_i$ is the projection onto $z_i$  holomorphically extendible $L^2(\partial D)$ functions (in $\bf C^2$, with a slightly more complicated formula in $\bf C^n$). 
\end{abstract}

\section{Introduction}       % Enter section title between curly braces
A fundamental question in function theory of several complex variables is to determine when a function defined on the boundary of a domain $D$ extends holomorphically to the domain. Unlike the case of the circle, there is no obvious small set of moments which suffices without redundancy.
In \cite{La}, the author showed that on  a 1-dimensional family of ellipsoids in $\bf C^2$ one can detect whether an $L^{\infty}$ function on $\partial D$ is CR by checking for holomorphic extensions on vertical and horizontal complex lines meeting $D$. This is a case of using the so-called "1-dimensional extension" property to check for  analytic extensions of functions.

Here is the CR Hartogs separate analyticity question, for domains in $\bf C^n$.

\begin{problem} Let $D\subseteq {\bf C^n}$ be a stricly convex domain. Let $f\in C(\partial D)$ be a continuous (resp.  $L^p$) function such for every (resp. almost every) affine 1-dimensional non-empty slice $L$  of $D$ parallel to a coordinate axis, $f|_{\partial D\cap L}$ extends holomorphically to $D\cap L$. Does this guarantee that $f$  is a $CR$ function---i.e. must $f$ have an extension to $D$ which is holomorphic in all variables?
\end{problem}

We call functions which satisfy the 1-dimensional extension property on the set of coordinate lines intersecting $D$ "CR Hartogs" functions, and write $f \in {\rm CRH}$. Other papers have considered this problem, or variants thereof, for continuous functions, while our result is for $L^p$ functions. 
To re-phrase Problem 1, we want to know if 
CRH=CR, for a given domain $D$. We remark that CRH=CR is false for the ball. In coordinates $(z,w)$, the function $f(z,w)=|z|^2$ is a CR Hartogs function, but $f$ is not CR. In \cite{La} the $L^2$ CRH functions on the sphere in $\bf C^2$ are classified.
The restriction to convex domains is natural, but may not be essential.

There are several papers which study a related problem on the ball in $\bf C^n$. Given a finite set of points $a_1, \dots, a_k \in B_n$, and given a function  $f$ on the boundary of $B_n$, one wants to show that if $f$ has holomorphic extensions on all the slices of $B_n$ by complex lines through the $a_i$'s, then $f$ is holomorphic. For a generic set of points in $B_n$, $k=n+1$ suffices. There are several versions and proofs of this kind of result in \cite{Bar}, \cite{KM}, 
\cite{Gl2}; 
in the third, the points are not restricted to be in the ball. 
We note that these theorems all deal with functions which are continuous.
It is natural to consider the case of $L^p$ functions with $p\ge 1$. By virtue of having the $p=2$ case of the theorem, we derive a new expression for the Szeg\H{o} kernel related to holomorphic extension on horizontal and vertical lines. 

There are many earlier theorems on testing holomorphic extension on lines. Important papers in this line of research were  \cite{St}, where the  family of all lines is shown to suffice for any $C^1$ bounded domain. In \cite{Di}, smaller families of lines were used; compared to our result, Dinh's theorem would require a real 3-dimensional set of lines for a domain in $\bf C^2$.  In a different direction, the very general    "parametric argument principle" of Agranovsky, \cite{Ag} is thematically related, but not directly applicable to the CR Hartogs problem.

In this paper we prove that CRH=CR for a large class of strictly convex domains with real analytic boundary in $\bf C^n$. Our method is based on duality and weak derivatives. The method of proof is outlined and a general theorem  proved in section 2; in sections 3 and 4, criteria for applying the main theorem are given. In section 5 we prove  by different means that the equation CRH=CR is valid under real analytic variations of a function and domain. Altogether, our results support the conjecture that the CR Hartogs theorem is generically true. 

\section{A general criterion for CRH=CR}
We deal with the case of domains in $\bf C^2$ first. Our criterion does not require domains with real analytic boundary; however, the verification for some domains uses the power series structure of the defining function. Let $D\subseteq \bf C^2$ be a convex smoothly bounded domain with some choice of smooth defining function $\rho$, where $d\rho\ne 0$ on $\partial D$. Suppose $f$ is a $C^1$ CRH function on $\partial D$. 
Let $F$ be the extension of $f$ to $D$ which is holomorphic on horizontal slices, and $G$ the extension which is holomorphic on vertical slices. Applying the tangential CR operator
$L=\rho_{\overline w}\frac{\partial}{\partial\overline z}
-\rho_{\overline z}\frac{\partial}{\partial \overline w}$ to both sides of the equation $F=G$ on $\partial D$, we obtain
\begin{equation}\label{eq1} \rho_{\overline w}F_{\overline z}=-\rho_{\overline z}G_{\overline w}.
\end{equation}

This equation can be re-written in the language of differential forms as
\begin{equation}\label{eq2} F_{\overline w}dz\wedge dw\wedge d{\overline w}=G_{\overline z}dw\wedge dz\wedge d{\overline z}\end{equation}
on $\partial D$. 

Denote the left side by $\omega$ and the right by $\eta$. Then by Stokes's theorem, it follows that $\int_{\partial D}p \omega=0$ for any  function  continuous $p$ which extends holomorphically to each horizontal slice, 
and $\int_{\partial D}q\eta=0$ for any continuous   $q$ which extends holomorphically to each vertical slice. From $F_{\overline w}dz\wedge dw\wedge d\overline{w}=LF\frac{dz}{\rho_{\overline z}}\wedge dw\wedge d{\overline w}=LG\frac{dz}{\rho_z}\wedge dw\wedge d{\overline w}=G_{\overline z}\frac{dw}{\rho_{\overline w}}\wedge dz\wedge d\overline z$, we see that $F_{\overline w}dz\wedge dw\wedge d{\overline w}$ is the product of a continuous function and the continuous  complex volume form, $dV=\frac{dz}{\rho_z}\wedge dw\wedge d{\overline w}=\frac{dw}{\rho_{\overline w}}\wedge dz\wedge d\overline z$. 

To show that $\omega=\eta=0$, it suffices to know that
\begin{equation}\label{eq3} C(\partial D)=\overline{P[z,w,\overline w]+P[z,\overline z,w]}, 
\end{equation}
where the closure of the polynomial algebras on the right side is taken in the uniform norm on $\partial D$.  

For any domain $D$ satisfying the above criterion, it follows that CRH=CR for $C^1$ functions. The next step is to use weak derivatives and an approximation lemma to prove the theorem for $L^p$ functions, with $p\ge 1$.
For background in CR extension, see \cite{Bog}.
Let $f\in L^p(\partial D)$. To show that the weak derivative $Lf=0$, it suffices to show that 
$$\int_{\partial D} (Lf)PdV=0$$ 
for any polynomial $P$. The splitting of the mixed anti-holomorphic terms must occur in the $C^1$ norm  on $\partial D$ for this step.

\begin{theorem}{\label{A}}
Let $D\subseteq \bf C^2$ be a strictly convex, smoothly bounded domain such that for any polynomial $g(z,w,\overline z,\overline w)$, there exist sequences of polynomials $p_n(z,\overline z,w)$ and $q_n(z,w,\overline w)$ such that 
$||g-p_n-q_n||_1\rightarrow 0$ as $n\rightarrow \infty$, where $||\cdot||_1$ is the $C_1$ norm on $\partial D$. Then CRH=CR for $L^p$ functions on $\partial D$, $p\ge 1$. 
\end{theorem}

\begin{proof}
First we show that $$\int_{\partial D}(Lf)p(z,w,\overline w)dV=0$$
and 
$$\int_{\partial D}(Lf)q(z,\overline z,w)dV=0$$ for any polynomials $p$  and $q$. 
Let $f_n\in L^p(\partial D)$ converge to $f$ in $L^p$, with $f_n$ extending holomorphically on almost every horizontal slice. We have 
$$\int_{\partial D}(Lf_n)p(z,w,\overline w)dV=0;$$
taking the limit, we get desired result for pairing with a polynomial $p(z,w,\overline w)$.
Taking  a sequence $g_n$ converging to $f$ in $L^p$ with holomorphic extensions  on almost all vertical slices gives the result for pairing with a polynomial $q(z,\overline z,w)$.

Next, we need to show that $$\int_{\partial D}(Lf)p(z,w,\overline z,\overline w)dV=0$$ for an arbitrary polynomial. This follows similar to above by choosing sequences $p_n(z,w,\overline w)$ and $q_n(z,\overline z, w)$ such that $p_n+q_n$ converges to $p(z,w,\overline z,\overline w)$ in the $C^1$ norm.   
\end{proof}

Although the criteria for a $C^1$ versus $L^p$ version of the CR Hartogs theorem are different, we only deal with the latter case in what follows, as our approximation method easily yields $C^1$ convergence. 

\section  {The case of a general ellipsoid}
The ellipsoid $D$ with defining function $\rho(z,w)=|z|^2+|w|^2-1+\epsilon(zw +\overline{zw})$ with $\epsilon >0$ 
was covered in the previous paper, where the equality CRH=CR was proved for $L^{\infty}$ functions.
$D$ has some  helpful special properties which are not generic: $D$ has cirles as vertical and horizontal cross-sections; also relevant is that $\rho_{\overline z}$ is holomorphic in $z$ and
$\rho_{\overline w}$ is holomorphic in $w$. 
If we take an ellipsoid with defining function $r(z,w)= |z|^2+|w|^2-1+\epsilon(zw +\overline{zw})+A(z^2+\overline z^2)+B(w^2+\overline w^2)$, ($A\ne 0, B\ne 0$), neither of the above properties hold.

We prove first the $L^p$  CR Hartogs theorem for a generic ellipsoid; the techniques developed here will be applied to prove the $L^p$ CR Hartogs theorem for a large class of domains. In the case of the ellipsoid, the decomposition of Theorem 1 is exact for polynomials, without approximation. 
\begin{theorem}{\label{B}}
Let $D$ be an ellipsoid with defining function
$r(z,w)=|z|^2+|w|^2 -1+\epsilon(zw+\overline{zw})+az^2 +\overline{az^2}+bw^2 +\overline{bw^2}$,
$(D=\{r<0\})$, with $0<\epsilon <\frac{1}{2}$. If $\epsilon-|a|-|b|>0$, then $D$ satisfies the hypotheses of Theorem 1; consequently, the $L^p$ CR Hartogs theorem holds for $D$. 
\end{theorem}

The  following lemma gives  the proof of the theorem for the ellipsoid.

\begin{prop}{\label {prop1}} Given the domain D from Theorem 2, there exists $\delta >0$ such that if $|a|,|b|<\delta$, any polynomial $P(z,w,\overline z, \overline w)$ can be decomposed as a sum $P=Q(z,w,\overline w)+R(z,\overline z, w)$, where $Q$ and $R$ are polynomials.

\end{prop}
\begin{proof}
For polynomials $p$ and $q$ (in holomorphic and anti-holomorphic variables) let "$p\equiv q$" mean that $p-q=[h(z,w,\overline w)+g(z,\overline z, w)]|_{\partial D}$  for some  polynomials $h$ and $g$. The lemma holds if  $\overline z^k\overline w^l\equiv 0$ for all $k,l \ge 0$. Note that if $A$ is a matrix with holomorphic coefficients, and if a vector valued function $v$ satisfies $v\equiv 0$, then $Av\equiv 0$. 
One need only show the lemma for monomials $\overline z^k\overline w^l$, $k>0, l>0$. We use induction on the total degree $k+l$. The case $k+l=2$ follows from the form of the defining function $r$, whose terms give the desired splitting for $\overline z\overline w$. Assume the lemma is true for $k+l=n-1$.

Multiply the equation $r=0$ by the $n-1$ antiholomorphic monomials of degree $n-2$:
$\overline z^{n-2}$, $ \overline z^{n-3}\overline w , \dots,\overline w^{n-2}$. By the inductive hypothesis, we need only consider the $n-1$ terms $\overline z^{n-1}\overline w, \dots,\\ \overline z\overline w^{n-1}$. We define $v_{n-1}=(\overline z^{n-1}\overline w, \overline z^{n-2}\overline w^2,\dots, \overline z \overline w^{n-1})^t.$ 

Set $$A_n=\begin{pmatrix} \epsilon & b&0  & \cdots & 0  & 0\\
                    a & \epsilon & b & \cdots &0 &0\\
                    \vdots &\vdots & \vdots &\ddots & \vdots & \vdots\\
                     0 &0  &0 & \cdots & a &\epsilon
                    \end{pmatrix}
                    .$$
                    
In particular $A_1=\epsilon$ and $$A_2=\begin{pmatrix}\epsilon & b\\ a &\epsilon \end{pmatrix}$$
Under the relation $\equiv$, the $n-1$ equations obtained from multiplying $r=0$ by the antiholomorphic monomials of degree $n-2$ can be written as $$A_{n-1}v_{n-1}\equiv 0.$$ 
To complete the proof of Theorem 2, we only need that $A_n$ is invertible; for the next stage, we need the estimates contained in the  following proposition.

\begin{prop}{\label{prop2}} Given $\delta >0$, there exists $\alpha >0$ such that if $|a|,|b|<\alpha$, then $||(A_n)^{-1}||<\frac{1+\delta}{\epsilon}$. The constants $a,b, \epsilon$ are from the defining function; $\alpha$ is independent of $n$. The norm $||.||$ is the operator matrix norm. 
\end{prop}
\begin{proof} Let $z=(z_1,\dots,z_n)^T$. $A_nz=\epsilon z+az^1 +bz^2$, where $z^1$ and $z^2$ are vectors consisting of the entries of $z$, shifted by 1, and with one omission.
From this, we can conclude that $|A_nz|\ge (\epsilon -|a|-|b|)|z|$ from which the lemma follows. 
\end{proof}
Since $A_n$  is invertible for all $n$, we see by induction that all of the holomorphic monomials are $\equiv 0$, which finishes the proof of Proposition 1.

\end{proof}
This ends the proof of Theorem 2.
\vskip .2in

For domains whose defining function has higher degree, analagous reasoning should  give an  exact decomposition of polynomial functions into sums of polynomials, each of which has a holomorphic extension in one of a finite set of directions. However, the approximation criteria of Theorem 1 does not require additional directions, in many cases.  To understand the principle behind the next theorem,  consider a cubic defining function
$\rho(z,w)=|z|^2+|w|^2-1 +\epsilon(zw +\overline{zw}) +az^2+\overline{az^2}+bw^2+\overline{bw^2}+cz^2w+\overline{c z^2w}$.

We have $\overline{zw}\equiv -\frac{c}{\epsilon}\overline {z^2w} $. 
If $\frac{c}{\epsilon}$ is small, we   can conclude  that the right side has a smaller uniform norm than that of $\overline{zw}$ on $\{\rho <0\}$. Following the induction scheme of Theorem 2, we might show that $\overline{zw}$ is in the uniform closure of the holomorphic in z plus holomorphic in w polynomials. In fact, this method works for a strictly convex domains with real analytic boundary, when the defining function is close enough to the defining function of the ellipsoid, in the appropriate sense.

We need an estimate for the norm of the vector $v_n$ defined above. 
\begin{lem}{\label{lemma 1}}
If $|z|^2+|w|^2=1$, then $|v_n|\le 1$.
\end{lem}
\begin{proof}

The lemma  follows from showing that $|u_n|\le 1$, where $u_n=(z^n, z^{n-1}w, ..,w^n)^t$.
We prove this by induction for $n\ge 2$. For  $n=2$ we have that for some $\theta$,
$$|u_n|=\cos^4(\theta)+\sin^2(\theta)\cos^2(\theta)+\sin^4(\theta)=
\cos^2(\theta)+\sin^4(\theta)\le 1.$$
The case for larger $n$ works similarly, where the norm of $u_n$ is dominated by the norm of some $\tilde{u}_{n-1}.$

\end{proof}

Let us now consider a perturbation of the ellipsoid by a power series. We group the terms by the anti-holomorphic degree. 
Write $\phi(z,w)=\Sigma_{n=0}^{\infty}\phi_n(z,w)$,
where $\phi_n$ is a power series of terms with anti-holomorphic degree $n$. We assume that the series for $\phi$ has no linear terms.  We suppose that the power series in $z,w,\overline z, \overline w$ converges for $|z|,|w|, |\overline z|, |\overline w|<r_1$, where $\sqrt{2}r_1<1$ (considering the conjugate arguments in the power series as separate variables) and that $\{\phi <0\}\subseteq \{|z|^2+|w|^2<{r}\}$ with $r<\frac{r_1}{\sqrt 2}$; the lack of linear terms makes this possible, by adjusting the constant. In the proof, one will see  that once the conditions with $r_1$ and $r$  are  satisfied, small enough linear terms can be allowed. This will be  relevant to the proof of Theorem 4.   The number $r$ can be made arbitrarily small, but in this proof we only use that $r<1$.   
We assume that $\phi_2$ is close enough 
$\epsilon\overline{zw} +\overline{az^2}+\overline{bw^2}$ that the matrix $A_n$ corresponding to the $\phi_2$ series satisfies 
$$||A_nv||\ge(1-\delta)||v||$$
for all $v$, with $0<\delta<1/2$, and for $|z|,|w|<r_1$. Write $\phi_1(z,w)=\phi_{10}\overline z+\phi_{11}\overline w$. We also require that $|\phi_{11}(z,w)|<r_2$,  $|\phi_{10}(z,w)|<r_2<1$, $\phi_0(z,w)|<r_2<1$ for $|z|, |w|<r_1$ , where $r_2$ can be made arbitrarily small (although here also we do not use anything more than $r_2<1$ in the proof).  
 We need to redefine "$\equiv$": now 
 $f\equiv g$ means that on $\partial D$, 
 $f-g=P(z,w,\overline w)+Q(z,w,\overline w)$, where $P$ and $Q$ are power series converging in a neighborhood of $\overline D$. 
For each $k\ge 0$ and each $n\ge 3$ we  construct a  $k+1$ by $n+k-1$ matrix $B_{k+1}^{n+k-1}$ whose rows are the coefficients of the degree $n+k+1$ mixed anti-holomorphic monomials obtained by multiplying $-{\phi_n}$ by the $k+1$ anti-holomorphic monomials of degree $k$ (for $k=0$, simply the equation $\phi=0$).
We estimate the norm by the same method used for the $A_n$'s.  $B_{k+1}^{n+k-1}$ has non-zero entries on $n+1$ diagonals only; therefore, we can find  a uniform bound $||B_{k+1}^{n+k-1}||<C_n$ for all $k$, depending only on the maximum of the coefficients of $\phi_n$. By choosing said coefficients small enough, one may arrange that $||B_l^m||<\epsilon^{4(m-l)}$ for $|z|,|w|<r_1$, for $l\ge 1$ and $m\ge 2$. 
\begin{theorem}{\label{C}}
With $\phi$ as above, the $L^p$  CR Hartogs theorem holds on the domain $\{\phi <0\}$ for small enough $\epsilon >0$.
\end{theorem}

%We write the equation $\phi(z,w)=0$ as  $\phi_1+\phi_2+\phi_3=\Sigma_{n=3}^{\infty} \overline{p_n(z,w)}$; $\phi_1$ is holomorphic; $\phi_2(z,w)=|z|^2+|w|^2$ and $\phi_3(z,w)=\epsilon(zw+\overline{zw})+a\overline z^2 +b\overline w^2$. 
  %On the left side, the only difference  is that the holomorphic part is the sum of all of the homogeneous holomorphic polynomials; this does not change the calculation of anti-holomorphic degrees. If $r<1$  is small enough, and if the higher order terms of $\phi$ are small enough,  then one may assume that for some (any, chosen once) $r_1<1$,   we have $|\phi_1(z,w)|<r_1$, $|z|<r_1$, $|w|<r_1$.
\begin{proof}
Let us consider the effect of multiplying $\phi=0$  by $\overline z^k$, $\overline z^{k-1}\overline w, \dots,\overline w^k$. 
As before, we want to gather all terms with mixed anti-holomorphic factors of degree $\le k+2$, use an inductive step to write them as having mixed anti-holomorphic degree exactly $k+2$, and then solve. 
We use the notation $()^*$ to denote a vector or matrix which is given as a submatrix of a larger matrix; in this case, 0 rows or entries are added to the matrix or vector so that dimensions in the equation match. Matrix operator norms will not change under the $*$ notation.      We get the vector equation 
\begin{equation}\label{eq4}
\phi_0 v^{*1}_{k-1} 
+\phi_{10}v^{*2}_{k}+\phi_{11}v^{*3}_{k}+A_{k+1}v_{k+1}\equiv 
\Sigma_{n=k+2}^{\infty} B_{k+1}^{n}v_n .
\end{equation}
Next, substitute \begin{equation}v_k\equiv(\tilde {A_k})^{-1}\tilde B_k^{k+1}v_{k+1}+
\Sigma_{n=k+2}^{\infty}(\tilde{A_k})^{-1}\tilde{B_k^n}v_n \end{equation}
and 
\begin{equation}v_{k-1}\equiv (\tilde A_{k-1})^{-1}\tilde B_{k-1}^k v_k+(\tilde A_{k-1})^{-1}\tilde B_{k-1}^{k+1}v_{k+1} +\Sigma_{n=k+2}^{\infty}(\tilde A_{k-1})^{-1}\tilde B_{k-1}^n v_n \end{equation}

\begin{equation}\equiv (\tilde A_{k-1})^{-1}\tilde B_{k-1}^{k}(\tilde A_{k})^{-1}\tilde B_{k}^{k+1}v_{k+1}+(\tilde A_{k-1})^{-1}\tilde B_{k-1}^{k+1}v_{k+1}$$ $$\hskip 1in +\Sigma_{n=k+2}^{\infty}\left( (\tilde A_{k-1})^{-1}\tilde B_{k-1}^n +
(\tilde A_{k-1})^{-1}\tilde B_{k-1}^k(\tilde A_k)^{-1}\tilde B_k^n\right)v_n.\end{equation}

The $\tilde{(\cdot)}$ matrices are the perturbations of the $A$'s and $B$'s defined inductively by substituting as above, and collecting terms. Assume that we have the following inequalities.
$$||(\tilde A_l)^{-1}||\le \frac{2}{\epsilon}, l\le k,$$

$$||(\tilde B_l^n)||\le 2 \epsilon^{4(n-l)}, \forall n >l, \forall l\le k.$$

We want to use induction to show that the same inequalities hold for $k+1$. For $k=0$, the hypotheses of the theorem give the inequality.  In this calculation, we only use that $|\phi_0|<1, |\phi_{10}|<1, |\phi_{11}|<1$. We calculate that the inductive hypothesis gives that the norm of $\tilde B_{k+1}^{n}$ is bounded  by $$\epsilon^{4(n-k-1)} +(\frac{2}{\epsilon})(2\epsilon ^{4(n-k+1)}) + 2(\frac{2}{\epsilon})2\epsilon^{4(n-k)} +(\frac{2}{\epsilon})(2\epsilon^4)(\frac{2}{\epsilon})2\epsilon^{4(n-k)},$$
which is less than $2\epsilon^{4(n-k-1)}$ if $\epsilon$ is small enough.  Similarly,
$$||(\tilde A_{k+1})v||\ge (1-\delta)\epsilon -\left(2(\frac{2}{\epsilon})2\epsilon^4 +(\frac{2}{\epsilon})2\epsilon^8 +(\frac{2}{\epsilon})(2\epsilon ^4)(\frac{2}{\epsilon})(2\epsilon^4)\right)||v|| \ge \frac{\epsilon}{2} ||v||,$$

 for all vectors $v$ if $\epsilon$ is small enough. 
Putting these estimates together, we conclude that for all $k$, there is a vector equation with holomorphic matrix coefficients, expressing $v_{k+1}$ in terms of the $v_l$, $l>k+1$:

\begin{equation}
v_{k+1}\equiv \Sigma_{n=k+2}^{\infty}(\tilde A_{k+1})^{-1}\tilde B_{k+1}^n v_n.
\end{equation}

On the right side, let us substitute the corresponding equation for $v_{k+2}$ and collect terms: 

$$v_{k+1}\equiv \Sigma_{n=k+3}^{\infty}
\left( (\tilde A_{k+1})^{-1}\tilde B_{k+1}^n+
(\tilde A_{k+1})^{-1}\tilde B_{k+1}^{k+2}(\tilde A_{k+2})^{-1}\tilde B_{k+2}^n\right)v_n.$$

The norm of the matrix coefficient of $v_n$ in this expression is bounded by 
$\epsilon^{4(n-k-2)}(16\epsilon^2+4\epsilon^3).$ 
The substitution for $v_{k+3}$ will result in norms bounded by $\epsilon^{4(n-k-3)}(16\epsilon^2+4\epsilon^3)(4\epsilon^2+\epsilon^4)$; in general, after repeating this process $l$ times, one gets matrix coefficients for $v_n$, starting at $n=k+l+2$, with norm  bounded by $\epsilon^{4(n-k-l-1)}(16\epsilon^2+4\epsilon^3)(4\epsilon^2+\epsilon^4)^{l-1}$. Since $||v_n||\le r^n$,
by choosing $l$ large enough, we can write
$v_{k+1}\equiv V_j$, where a geometric series estimate shows that  $||V_j||\rightarrow 0$ as $j\rightarrow \infty$. 
This proves that the components of $v_{k+1}$, which are the degree $k+1$ anti-holomorphic monomials, are in
$\overline{P[z,w,\overline w]+Q(z,w,\overline z)}$ on $\partial D$.
By the choice of constants and the properties of $\phi$, the convergence  of $V_j$ holds for $|z|,|w|,|\overline z|, |\overline w|<r_1$.
A fortiori, this implies $C^1$ convergence  to 0 on $\partial D$; this finishes the proof of the $L^p$ CR Hartogs theorem for $\bf C^2$. 

\end{proof}

In the proof, norms of the holomorphic in $z$ and holomorphic in $w$ terms are not estimated. At each stage, this remainder consists of some terms in a convergent power series, divided by a small but non-zero denominator. This gives the convergence at each stage, which is all that is needed.
Estimating this part precisely may lead to sharper  results in the future. 

As a corollary of our main theorem, we derive a novel expression for the Szeg\H{o} kernel as an iterative limit. First we prove a lemma about intertwining Hilbert space projections.

\begin{lem} Let $H$ be a Hilbert space, with $V\subseteq H$ and $W\subseteq H$ two closed subspaces. Let $\pi_1$ and $\pi_2$ be the associated orthogonal projections, $\pi$ the projection onto $V\cap W=U$. Then for any $v\in V$,  $
lim_{\rightarrow\infty}(\pi_1 \pi_2)^n(v) =\pi(v)$. 
\end{lem}

\begin{proof} The actions of $\pi_1$ and $\pi_2$ both split on $U\oplus U^{\perp}$ with the action on the first factor the identity. Therefore after taking the quotient by $U$, we may assume that $V\cap W=\{0\}$. We analyze a slightly different limit which implies the desired result.
The operator $T=\pi_2\pi_1\pi_2$ is self-adjoint, and $T^n=\pi_2(\pi_1\pi_2)^n$. By the spectral theorem, $T$ is unitarily equivalent to a multiplication operator on a finite measure space $(X,\Omega,d\mu)$. The multiplication operator will be given by a function $f(x)$, $||f||_{\infty}\le 1$ and $\mu(f^{-1}(\{1\}))=0$, since $||T(v)||<1$ for all $v$. In this form, we see by the dominated convergence theorem that 
$T^n(v)\rightarrow 0$ for all v.

\end{proof}

\begin{cor}
Suppose  $D\subseteq \bf C^2$ is a domain satisfying the hypotheses of Theorem 3. Set $H=L^2(\partial D)$, and let $V_1$ (resp. $V_2$) be the subspaces consisting of functions with holomorphic extension on  almost all horizontal (resp. vertical) slices. Let $\pi_i$ be the orthogonal projection of $H$ onto $V_i$. Then 
$(\pi_1 \pi_2)^n\rightarrow S$ in the strong operator topology, where $S$ is the Szeg\H{o} projection. 

\end{cor}

\section{Convex domains in $\bf C^n$}
The case of $\bf C^3$ illustrates what happens in all higher dimensions; therefore, we explain the strategy for $\bf C^3$ and then formulate the conditions for the general theorem in $\bf C^n$.
Let $D\subseteq \bf C^3$ be a smoothly bounded strictly convex domain with smooth defining function $\rho$, and let $f$ be a CR Hartogs functions on $\partial D$.
Let $L_1=\rho_{\overline z_2}\frac{\partial}{\partial \overline z_1}-\rho_{\overline z_1}\frac{\partial}{\partial \overline z_2}$, 
$L_2=\rho_{\overline z_3}\frac{\partial}{\partial \overline z_1}-\rho_{\overline z_1}\frac{\partial}{\partial \overline z_3}$,
$L_3=\rho_{\overline z_3}\frac{\partial}{\partial \overline z_2}-\rho_{\overline z_2}\frac{\partial}{\partial \overline z_3}$. To show that $f$ is CR, it suffices to show that the weak derivative $L_i(f)=0$ on $\partial D$ for two of the $L_i$'s, since the complex tangent space of $\partial D$ is two-dimensional at all points. We modify the condition of Theorem 1 to allow for more variables as follows. The condition for $z_1$ and $z_2$, associated with the operator $L_1$ is given; the other conditions are given by switching variables.

\hskip .3in {\it Every $g\in C^1(\partial D)$ is in 
$\overline{P[z_1,z_2,\overline z_2, z_3,\overline z_3]+P[z_1,\overline z_1,z_2, z_3,
\overline z_3]}$, where the closure is in the $C^1$ norm. }

If two of these conditions hold, then the $L^p$ CR Hartogs theorem holds for $D$. The verification proceeds by power series calculations, as in Theorem 3. We use $A_n$, $B_k^n$, etc., to denote the matrices from the proof of Theorem 2, with respect to two variables; the third is treated as a parameter.

%and any $\epsilon >0$, there exist power series $P(z_1,z_2,\overline z_2, z_3, \overline z_3)$ and \\ $Q(z_1,\overline z_1, z_2, z_3, \overline z_3)$, converging in a neighborhood of $\overline D$ such that $|g-P-Q|<\epsilon$ on $\partial D$. 
%\item For every $g\in C(\partial D)$ and any $\epsilon >0$, there exist power series $P(z_1,z_2,\overline z_2, z_3, \overline z_3)$ and \\$Q(z_1,\overline z_1, z_2, \overline z_2, z_3)$, converging in a neighborhood of $\overline D$, such that $|g-P-Q|<\epsilon$ on $\partial D$. 
%\item For every $g\in C(\partial D)$ and any $\epsilon >0$, there exist power series $P(z_1, \overline z_1,z_2, z_3, \overline z_3)$ and\\ $Q(z_1,\overline z_1, z_2, \overline z_2, z_3)$, converging in a neighborhood of $\overline D$, such that $|g-P-Q|<\epsilon$ on $\partial D$. 

Let us consider an ellipsoid in $\bf C^3$. Now $\epsilon_i$ are small complex numbers. 
The defining function can be written as \begin{equation}\rho(z_1,z_2,z_3)=|z_1|^2+|z_2|^2+|z_3|^2-r+Re(\epsilon_1 z_1z_2 +\epsilon_2 z_1z_3 +\epsilon_3 z_2z_3) +\Sigma_{i=1}^3 Re(a_i z_i^2),
\end{equation}
Write this equation with one of the variables, say $z_3$, treated as a parameter:  $\rho=\rho_0+\rho_1+\rho_2$, where $\rho_i$ is anti-holomorphic of degree $i$ in $z_1$ and $z_2$, 
We need to know that $\rho_0$, $\rho_{10}$, $\rho_{11}$ are small enough that the $\tilde A_n$'s still have inverses with bounded norm. Given any choice of $\epsilon_i$'s, this will occur if $r$ and $a_i$, $i=1,2,3$ are small enough. Thus,   the conditions of Theorem 2 can hold for many power series perturbations of the defining function for an ellipsoid. 
Lemma 1 will also be true: a vector $v_n$ created using two out of three variables with $|z_1|^2+|z_2|^2+|z_3|^2=1$ will also satisfy the hypotheses of Lemma 1. Putting these observations together, we can now state the CR Hartogs theorem for $C^n$.
We use $\alpha, \overline{\beta}$ to denote holomorphic and anti-holomorphic multi-indices. 
Given indices $1\le i<j\le n$, $\tilde{\alpha}, \overline{\tilde{\beta}}$ will denote an index in the variables $z_1$ to $z_n$ where the sum of the holomorpic in $z_i$, $z_j$ degrees and either holomorphic or anti-holomorphic degrees in the remaining variables is $\alpha$; of anti-holomorphic degree $\beta$ in $z_i$ and $z_j$. 
\begin{theorem}
Let $\rho(z)=\Sigma_{|\alpha|\ge 0,|\overline \beta|\ge 0}z^{\alpha}\overline z^{\beta}$, $z=(z_1,\dots,z_n)$ converge for $|z|<r^{\prime}$, with $\{\rho <0\}\subseteq \{|z|<r\}$, $r<r^{\prime}$. Define $L_{ij}=\rho_{\overline z_j}\frac{\partial}{\partial \overline z_i}-\rho_{\overline z_i}\frac{\partial}{\partial \overline z_j}$. Suppose that for  subcollection    $L_{i_1j_1}, \dots, L_{i_{n-1}j_{n-1}}$ which is linearly independent on $\partial D$, the expansions $\rho(z)=\Sigma_{|\tilde\beta|=k=0}^{\infty}\rho_k$ satisfy the estimates  of Theorem 2 in the $z_{i_m}$ and $z_{j_m}$ variables. Then a  function $f\in L^p({\partial D})$, $1\le p\le \infty$ which is CR Hartogs is a $CR$ function whose extension is in the Hardy space $H^p(D)$. 
\end{theorem}
The fact that the extension is in $H^p$ follows from the fact that each a priori extension  which is holomorphic in one of the variables, is in $H^p$ on each slice. 

We also have an asymptotic formula for the Szeg\H{o} projection S. Let $\pi_i$, $i=1,\dots, n$ be the projections onto the $z_i$ holomorphically extendible functions, as in Corollary 1. 
\begin{cor}
Let $D\subseteq \bf {C^n}$ be a domain where the $L^p$ CR Hartogs theorem is true. Set 
$T=\pi_n\pi_{n-1}\dots \pi_1\dots \pi_n$. Then 
$T^k\rightarrow S$ in the strong operator topology. 
\end{cor} 

\section{Variation of the CR Hartogs phenomemon}
Our main theorem, while very general, leaves open the question of whether CRH=CR is generically true. In this section, we use a new method to show that the relation CRH=CR is stable under real analytic variations of a domain and function. Because of the strong conditions on the defining function in Theorems 3 and 4, it is not clear whether a real analytic perturbation of $D$  gives a new domain where CRH=CR. We hope that the method here can be strengthening to show stability of the $L^p$ CR Hartogs theorem. 
We use the integral form of the extension conditions,
\begin{equation}\int_{\partial D} f(z,w)p(z,\overline{z}, w) dw\wedge dz\wedge d\overline{z}=0
\end{equation}
 for all polynomials $p$ in the variables $z$, $\overline z$ and $w$.
For extension on all horizontal slices, the condition is 
\begin{equation}
\int_D f(z,w)q(z,w,\overline{w})dz\wedge dw\wedge d\overline{w}=0
\end{equation}
 for all $q(z,w, \overline{w})$. 
This proof makes no use of exterior algebra; hence, we will write $dH=dz\wedge dw\wedge d\overline{w}$ and $dV=dw\wedge dz\wedge d{\overline z}$. When there are variations, we will write add a $t$ subscript to denote a dependence of the volume form on $t$.

Suppose that $D_t$ is a real analytic variation of  a domain $D_0$. By this we mean that there is a family $D_t$ for $|t|<\delta$ with $\partial D_t$ given as a real analytic graph in the normal bundle of $\partial D_0$ realized as a neighborhood of $\partial D_0$.  Let $U$ be a neighborhood of $\partial D_0$ and suppose that $\cup_t \partial D_t \subseteq U$. Let $f_t=f(t,z,w) \in C^{\omega}(U\times (-a,a))$,  be real analytic in $(z,w,t)$, and  suppose that  the restriction of $f_t$ to $\partial D_t$ is a CR Hartogs function  for all $t$.  This is what we mean by a real analytic variation of CR Hartogs functions.  We are interested in the case where $CRH=CR$ on $D_0$.  Here is our result.

\begin{theorem}\label{E}
Let $f_t$ be a real analytic variation of CR Hartogs functions on the real analytically varying domains $D_t$, $|t|<\delta$. If CRH=CR on $D_0$, 
then every $f_t$ is a CR function on $\partial D_t$.
\end{theorem}
%Let $\eta$ be a (real) $C^{\omega}$ vector field defined on $\partial D_0$. 
By assumption,  every $\partial D_t$ can be written as a graph over $\partial D_0$ in the normal coordinate. 
Let $\iota_t:\partial D_0 \rightarrow \partial D_t$ be the function which sends a point $x\in \partial D_0$ to the point in $D_t$ lying in the normal through $x$. We can use $\iota_t$ to pull back the CR Hartogs data to $\partial D_0$. Applying this to (10) and (11), we get 
\begin{equation}\label{3}
\int_{\partial D}f(\iota_t(z,w))p(\iota_t(z,\overline{z},w)) \iota_t^*(dV|_{\partial D_t})=0
\end{equation}

and \begin{equation}\label{4} 
\int_{\partial D}f(\iota_t(z,w))q(\iota_t(z,w, \overline w))\iota_t^*(dH|_{\partial D_t})=0
\end{equation}
Writing $f_t(z,w)=f(\iota_t(z,w))$, $p_t(z,\overline z, w)=p(\iota_t(z,\overline{z},w))$, $q_t(z,w,\overline w)=q(\iota_t(z,w,\overline w))$,\\
$dV_t=\iota_t^*(dV|_{\partial D_t})$, $dH_t=\iota_t^*(dH|_{\partial D_t})$,  we get the abbreviated form
\begin{equation}\label{5}
\int_{\partial D}f_t(z,w){p_t}(z,\overline z, w) dV_t=0
\end{equation}
and 
\begin{equation}\label{6}
\int_{\partial D}f_t(z,w)q_t(z, w,\overline w) dH_t=0
\end{equation}

Differentiating (14) and (15) repeatedly with respect to $t$ and setting $t=0$, we get an infinite sequence of equations for each polynomial $p$ and $q$. These equations become extremely complicated quite quickly, but by utilizing the fact that $CRH=CR$ for $D_0$, we can effect a significant simplication.
Note that $f_0$ is a $CR$ function on $\partial D_0$. Being real analytic and $CR$, it has a holomorphic extension to a neighborhood $U_1\subseteq U$ of $\partial D_0$. Let $g^0$ be this function, and let $g^0_t$ be the pullback of $g^0|\partial D_t$ to $\partial D_0$ by the normal projection. 

Let $F^0_t=f_t-g^0_t$. $F^0_t$ is CRH for all $t$. Taking the $t$ derivative of (14) and (15) (using dot notation) and plugging in $t=0$, and using $F^0_0(z,w)\equiv 0$ we get

\begin{equation}\label{7}
\int_{\partial D}\dot{F^0}|_{t=0}(z,w)p(z,\overline z, w) dV=0
\end{equation} for all polynomials $p(z,\overline z, w)$
and 
\begin{equation}
\int_{\partial D}\dot{F^0}|_{t=0}(z,w)q(z,w, \overline w) dH=0
\end{equation}
for all polynomials $q(z,w,\overline w)$.
Because of the assumption about $D_0$, this gives that $\dot{F^0}|_{t=0}$ is a $CR$ function; denote this function by $g^1$, and extend it to  a neighborhood $U_2\subseteq U_1$ of $\partial D_t$; set $F^1_t=F^0_t-tg^1_t$.
Then $F^0_0=\dot{F^1}|_{t=0}=0$ on $\partial D$. Taking two derivatives of (14) and (15) and setting $t=0$, we conclude that $\dot{\dot{F^1}}|_{t=0}$ is a $CR$ function.
Continuing this process, we arrive at the following result. We continue with the convention that a variation $h_t$ is considered as a function on $\partial D_0$, via the normal pullback. 
\begin{lem} Let $f_t$ be the pull back to $\partial D_0$ of a real analytic variation of CR Hartogs functions. Then for any n, there exist $g^0$, $g^1,\dots, g^n$ which are real analytic $CR$ functions on $\partial D_0$ and there exists $\epsilon>0$ such that $g^i_t$ is defined for $|t|<\epsilon$ and 
\begin{equation}
F^n_t=f_t-g^0_t-tg^1_t\dots-\frac{t^ng^n_t}{n!}
\end{equation} vanishes to order $n$ in $t$ at $t=0$ for $(z,w)\in \partial D_0$.

\end{lem}
At this point, one would like to say that $f$ has a power series expansion in functions which are $CR$ in a a fixed neighborhood of $\partial D_0$; however, there does not seem to be an easy way to control the neighborhoods where the $g^i$ extend holomorphically past $\partial D_0$.
Instead, we use the variation of the Bochner-Martinelli projection to show that $f$ is $CR$ on each $\partial D_t$, for small enough $t$.
\begin{claim} Let $B_t$ be the pullback of the Bochner-Martinelli projection on $\partial D_t$. Then for any real analytic variation  $f_t$(not necessarily CR Hartogs) there exists $\delta >0$ such that  $B_t(f_t)$ is a real analytic  function of $(z,w,t)$ for  $(z,w)\in \partial D_0$, $|t|<\delta$.
\end{claim}
The Bochner-Martinelli kernel, while not holomorphic, does have the property that the boundary operator reproduces exactly the CR functions. This is proved in, for example \cite{KyAi}. 

\begin{theorem}
Let $D$ be a domain with $C^2$ boundary, and let $B$ denote the Bochner-Martinelli operator on $\partial D$.
Then for a continuous function $f$, $Bf=f$ if and only if $f$ is CR.

\end{theorem}

Assuming the claim, we can prove that the CR Hartogs variation  is in fact a CR variation. 

Let $H(z,w,t)=(I-B_t)f_t$. We have that $H=(I-B_t)(F^n_t)$ for all $n$, because the difference is a CR variation. Since $F_n$ vanishes to order $n$ at $t=0$ on $\partial D$, we have that 
$\frac{d^k}{dt^k}|_{t=0}H(z,w,t)\equiv 0$ for $k=0, 1,\dots,n$.  Letting $n\rightarrow \infty$, we get that $H\equiv 0$---i.e., the CR Hartogs variation is actually a variation of CR functions. 

Now we proceed to the proof of the claim.
We use the expression of the Bochner-Martinelli kernel as the average of the Cauchy transform on lines; see e.g. \cite{Kyt}. Let $K(z,w)$ denote the Bochner-Martinelli kernel.  For a point $z\in D$, 
\begin{equation}
\int_{\partial D} f(z)K(z,w)=\int_{{\bf CP^1}_z}\left( \int_{L\cap \partial D}C_L(f)\right),
\end{equation}
where $C_L$ is the Cauchy integral on $L\cap \partial D$, and ${\bf CP^1}_z$ is the set of complex lines through $z$. The Bochner-Martinelli operator on the boundary is a singular integral operator.  Real analyticity allows one to sidestep issues of the principal value of a Cauchy transform, as shown  below.  

We need some facts on real analyticity of the Cauchy transform.
\begin{lem}
Let $f(z,t)$ be real analytic on $(|z|=1)\times(|t|<\delta)$, where $t\in \bf R^n$. Let $Cf$ be the Cauchy transform with respect to $z$. Then for any $\delta_1<\delta$, there exists  $r>1$ such that  $Cf(z,t)$ is real analytic on 
$(|z|<r)\times (|t|<\delta_1)$.
\end{lem}
\begin{proof}
Given $\delta_1$, we  can use compactness to find  $r_1<1<r_2$ so that $f$ can be extended to a holomorphic in $z$ function on $(r_1<|z|<r)\times(|t|<\delta_1)$. Expressing the Cauchy transform on $|z|=1$ in terms of the  Cauchy integral on $|z|=\frac{1+r_2}{2}$ gives the desired real analyticity. 
\end{proof}

It is  classical  that the Riemann map varies real analytically with a real analytic variation of a domain; see \cite{BG} for a recent paper which includes a proof of this fact.
\begin{lem} Let $\phi(z,\tau):(|z|<r)\times (|\tau|<\delta_1)\rightarrow \bf C$ ($r>1$) be a real analytic function which for fixed $\tau$ is holomorphic, and a conformal map on $|z|\le 1$. Let $U\subseteq \bf C$ be a neighborhood of 
$\overline{\phi((|z|<r)\times(|\tau|<\delta))}$, and suppose $f(z,\eta)$ is a real analytic function on $U\times(|\eta|<\delta_2)$. For each $\tau$, let $C_{\tau}(f)(z)$ be the Cauchy transform of $f$ on $\phi(\cdot,\tau)(|z|=1)=\gamma_{\tau}$.
Then   $C_{\tau}(f)$ is a real analytic function of $(z, \tau,\eta)$ on $(|z|=1)\times(|\tau|<\delta_1)\times (|\eta|<\delta_2)$.
\end{lem}
\begin{proof}
Pick $r_1$, $1<r_1<r$ so that $\phi(z,\tau)$ is 1-1 on $|z|\le r_1$ for each $\tau$. 
The Cauchy transform on $\gamma_{\tau}$ pulls back to $|z|=1$ as
$$\frac{1}{2\pi i}p.v. \int_{|\zeta|=1}\frac{f(\zeta, \eta)d\zeta}{\phi(\zeta,\tau)-\phi(z,\tau)}.$$
Write $\phi(\zeta,\tau)-\phi(z,\tau)=h(z,\zeta,\tau)(\zeta-z)$, where $h$ is real analytic and non-vanishing. 
Apply Lemma 2 to finish the proof. 
\end{proof}

Let $D\subseteq \bf C^2$ be a $C^{\omega}$ strictly convex domain. With an affine change of coordinates, we may arrange that $0\in \partial D$, and that $T_0(\partial D)=\{v=0\}$, where coordinates are $z=x+iy$ and $w=u+iv$. 
Given any point $p\in \partial D_t$ with $t$ and $p$ small enough, there is a complex affine change of coordinates sending $p$ to $0$ and $T_p(\partial D_t)$ to $\{v=0\}$. This 
 transformation can be taken to be $C^{\omega}$ in $t$ and $p$.
Set $\tau=(x,y,u)$. For small enough $t$, $\tau$ is a local coordinate on $\partial D_t$. 
The real analyticity of $B_t(f_t)$ will follow if we show that the Bochner-Martinelli transform at  $0$ varies real analytically in $t, \tau$ for slices which are close to $w=0$ (the transverse slices are easily dealt with).  
To show that the Bochner-Martinelli transform of $f_t$ varies real analytically, it is  to show that the Cauchy transform on each slice through $0$  varies real  analytically in $\tau, t$, with bounds. In this case, bounds follow from real analyticity at 0 of a related object.

 We consider the local picture.
In coordinates $z=x+iy$ and $w=u+iv$, we may suppose that near 0,  $\partial D$ is given by the graph of a function $v=\phi(x,y,u)=a(t)x^2+b(t)y^2+c(t)u^2+d(t)ux+e(t)uy+f(t)xy +O(|x|,|y|,|u|)^3$, where the last term depends on $t$ as $O(1)$. Consider the slice of $\partial D_t$ by the complex line $w=\alpha z$ where by rotation we may assume that $\alpha$ is real.
 By the assumption of strict convexity, this   is a graph over a strictly convex curve $C_{\alpha, t}$.  Dilate by $x=\alpha r$, $y=\alpha s$ to get a new set of curves $\tilde{C_{\alpha,t}}$ which are a real analytic deformation of the $\alpha=0$ limit, which depends only on the quadratic part of the defining function. We remark here that the Cauchy transform operator on a curve commutes with complex affine transformations of $\bf C$. Therefore we can compute the Cauchy transforms of $f_t$ on the dilated curves. For small enough $\delta$ and $\epsilon$, there is a map $\psi: [0,\epsilon)\times(-\delta,\delta)\times S^1 \rightarrow \bf C$  such that $\psi(k,t,\cdot)$ parametrizes 
$\tilde{C_{\alpha,t}}$ and is real analytic at $k=0$. By the remark about real analytic variations of conformal maps, we may assume that $\psi(k,t,\cdot)$ is the boundary function for a conformal map with fixed base point.  The function $f_t$ on the slices $w=k\alpha z$ slices of $\partial D_t$ pulls back to to a real analytic function on $[0,\epsilon)\times(-\delta,\delta)\times S^1$. 
Finally, the real analyticity of the Cauchy transform  in $t$ and $\tau$ follows from Lemma 5. 
The map $\psi$ obviously varies real analytically as the argument of $\alpha$ changes. This is sufficient to prove the real analyticity of $B_t(f_t)$. 

We remark that this method can be adapted to show  real analytic stability of the theorems of Baracco, et. al. The goal should be to prove that 
theorems of the CR Hartogs type are generically true for $L^P$ functions. At present there is no evidence for the case of boundaries which are not real analytic. The present work sheds no light on  this question.

\end{document}